\documentclass[11pt]{article}

\usepackage{amsmath,amssymb,amsfonts,amsthm}

\usepackage{tikz}
\usepackage{graphicx}
\usepackage{placeins}
\usetikzlibrary{automata}
\usetikzlibrary{arrows}
\usetikzlibrary{positioning,calc}
\usetikzlibrary{graphs}
\usetikzlibrary{graphs.standard}
\usetikzlibrary{arrows,decorations.markings}
\usepackage{tkz-graph}
\usetikzlibrary{chains,fit,shapes}
\usetikzlibrary{calc}
\tikzset{every loop/.style={min distance=10mm,looseness=10}}
\tikzset{every state/.style={minimum size=2mm}}

\newtheorem{theorem}{Theorem}

\newtheorem{remark}[theorem]{Remark}

\newtheorem{conjecture}{Conjecture}

\title{A note on semi-transitivity of Mycielski graphs}

\author{Sergey Kitaev\footnote{Department of Mathematics and Statistics, University of Strathclyde, 26 Richmond Street, Glasgow G1, 1XH, United Kingdom. 
{\bf Email:} sergey.kitaev@strath.ac.uk.}\ \ and Artem Pyatkin\footnote{Sobolev Institute of Mathematics, Koptyug ave, 4, Novosibirsk, 630090, Russia}\ \footnote{Novosibirsk State University, Pirogova str. 2, Novosibirsk, 630090, Russia. {\bf Email:} artem@math.nsc.ru.}}

\begin{document}
	\maketitle

\begin{abstract}
An orientation of a graph is semi-transitive if it contains no directed cycles and has no shortcuts. An undirected graph is semi-transitive if it can be oriented in a semi-transitive manner. The class of semi-transitive graphs includes several important graph classes. The Mycielski graph of an undirected graph is a larger graph constructed in a specific manner, which maintains the property of being triangle-free but increases the chromatic number.

In this note, we prove Hameed's conjecture, which states that the Mycielski graph of a graph $G$ is semi-transitive if and only if $G$ is a bipartite graph. Notably, our solution to the conjecture provides an alternative and shorter proof of the Hameed's result 
on a complete characterization of semi-transitive extended Mycielski graphs.	\\
	
		\end{abstract}	

\section{Introduction}
{\em Semi-transitive graphs}, also known as {\em word-representable graphs}, 
include several fundamental classes of graphs (e.g.\ {\em circle graphs}, {\em $3$-colorable graphs} and {\em comparability graphs}), and they have being the subject of much research in the literature \cite{KL15}. In particular, Hameed~\cite{Hameed} provided a complete classification of semi-transitive {\em extended Mycielski graphs} and also {\em Mycielski graphs of comparability graphs}. Our goal in this note is to settle a conjecture of Hammed~\cite{Hameed}  by giving a complete classification of semi-transitive Mycielski graphs. 

\subsection{Main definitions}
 The {\em Mycielski graph} of an undirected graph is a larger graph that preserves the property of being triangle-free but enlarges the chromatic number. These graphs were introduced by Mycielski in 1955 (see \cite{M}) to prove the existence of  triangle-free graphs with arbitrarily large chromatic number. Since its introduction, Mycielski graphs have attracted considerable attention in the literature from various perspectives; see, e.g., \cite{B} and references therein.
 
Let the set of vertices in a graph $G$ be $\{1, 2, \ldots , n\}$. The $Mycielski$ graph $\mu(G)$ contains $G$ itself as a subgraph, together with $n+1$ additional vertices: a vertex $i'$ corresponding to each vertex $i$ of $G$, and an extra vertex $x$. Each vertex $i'$ is adjacent to $x$, forming a star subgraph $K_{1,n}$. Additionally, for each edge $ij$ in $G$, the $Mycielski$ graph includes two edges: $i'j$ and $ij'$. For instance, the graph $\mu(C_{5})$ also known as Gr\"otzsch graph is presented in Figure~\ref{Mycielski} to the left. Here and throughout the note, $C_n$ denotes the cycle graph on $n$ vertices labelled around the cycle $1,2,\ldots,n$.

\begin{figure}
\begin{tabular}{c c c}

\hspace{1cm}		
\begin{tikzpicture}[scale=0.5]
			
\draw (2,8) node [scale=0.5, circle, draw](node1){$1$};
\draw (4,8) node [scale=0.5, circle, draw](node2){$2$};
\draw (6,8) node [scale=0.5, circle, draw](node3){$3$};
\draw (8,8) node [scale=0.5, circle, draw](node4){$4$};
\draw (10,8) node [scale=0.5, circle, draw](node5){$5$};

\draw (2,5) node [scale=0.5, circle, draw](node6){$1'$};
\draw (4,5)node [scale=0.5, circle, draw](node7){$2'$};
\draw (6,5)node [scale=0.5, circle, draw](node8){$3'$};
\draw (8,5)node [scale=0.5, circle, draw](node9){$4'$};
\draw (10,5)node [scale=0.5, circle, draw](node10){$5'$};
\draw (6,3)node [scale=0.5, circle, draw](node11){$x$};

\draw (node1)--(node2) -- (node3) -- (node4) -- (node5);
\draw [bend right=20] (node5) to (node1);
\draw (node1)--(node7);
\draw (node2)--(node8);
			
\draw (node3)--(node9);
\draw (node4)--(node10);
\draw (node6)--(node2);
\draw (node7)--(node3);
\draw (node8)--(node4);
\draw (node9)--(node5);
\draw (node11)--(node6);
\draw (node11)--(node7);
\draw (node11)--(node8);
\draw (node11)--(node9);
\draw (node11)--(node10);
\draw (node1)--(node10);
\draw (node6)--(node5);
\end{tikzpicture}

&

&

	\begin{tikzpicture}[scale=0.5]
		
		\draw (2,8) node [scale=0.5, circle, draw](node1){$1$};
		\draw (4,8) node [scale=0.5, circle, draw](node2){$2$};
		\draw (6,8) node [scale=0.5, circle, draw](node3){$3$};
		\draw (8,8) node [scale=0.5, circle, draw](node4){$4$};		
		\draw (10,8) node [scale=0.5, circle, draw](node5){$5$};
		\draw (2,5) node [scale=0.5, circle, draw](node6){$1'$};
		\draw (4,5)node [scale=0.5, circle, draw](node7){$2'$};
		\draw (6,5)node [scale=0.5, circle, draw](node8){$3'$};
		\draw (8,5)node [scale=0.5, circle, draw](node9){$4'$};
		\draw (10,5)node [scale=0.5, circle, draw](node10){$5'$};
		\draw (6,3)node [scale=0.5, circle, draw](node11){$x$};				
		\draw (node1)--(node2) -- (node3) -- (node4) -- (node5);
		\draw [bend right=20] (node5) to (node1);		
		\draw (node1)--(node7);
		\draw (node1)--(node8);
		\draw (node1)--(node9);
		\draw (node1)--(node10);
		\draw (node2)--(node6);
		\draw (node2)--(node8);
		\draw (node2)--(node9);
		\draw (node2)--(node10);
		\draw (node3)--(node6);
		\draw (node3)--(node7);
		\draw (node3)--(node9);
		\draw (node3)--(node10);
		\draw (node4)--(node6);
		\draw (node4)--(node7);
		\draw (node4)--(node8);
		\draw (node4)--(node10);
		\draw (node5)--(node6);
		\draw (node5)--(node7);
		\draw (node5)--(node8);
		\draw (node5)--(node9);
		\draw (node5)--(node9);
		\draw (node6)--(node11);
		\draw (node7)--(node11);
		\draw (node8)--(node11);
		\draw (node9)--(node11);
		\draw (node10)--(node11);		
	\end{tikzpicture}
			
\end{tabular}
	
\caption{The graphs $\mu(C_5)$ (to the left) and $\mu'(C_5)$ (to the right)} \label{Mycielski}
	
\end{figure}
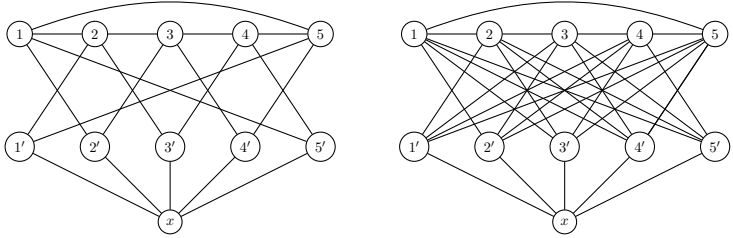

The {\em extended Mycielski graph} $\mu'(G)$ is derived from a Mycielski graph $\mu(G)$ by connecting every vertex $i'$ to every vertex $j$, except for $j=i$. The graph $\mu'(C_{5})$ is presented in Figure~\ref{Mycielski} to the right.  Note that for a complete graph $K_n$, $\mu'(K_n)=\mu(K_n)$. 


An orientation of a graph is {\em semi-transitive} if it is acyclic (there are no directed cycles), and for any directed path $v_0\rightarrow v_1\rightarrow \cdots \rightarrow v_k$ either there is no edge between $v_0$ and $v_k$, or $v_i\rightarrow v_j$ is an edge for all $0\leq i<j\leq k$. 
An induced subgraph on vertices $\{v_0,v_1,\ldots,v_k\}$ of an oriented graph is a {\em shortcut} if its orientation is acyclic and non-transitive, and there is the directed path $v_0\rightarrow v_1\rightarrow \cdots \rightarrow v_k$  and the edge $v_0\rightarrow v_k$ called the shortcutting edge. Note that any shortcut must contain at least 4 vertices. A semi-transitive orientation can then be alternatively defined as an acyclic shortcut-free orientation. A non-oriented graph is {\em semi-transitive} if it admits a semi-transitive orientation.

A {\em source} (resp., {\em sink}) in a directed graph is a vertex with ingoing (resp., outgoing) edges. The following result proven in \cite{KS} is quite helpful.

 \begin{theorem}[\cite{KS}]\label{source-thm} Suppose that a graph $G$ is semi-transitive, and $v$ is a vertex in $G$. Then there exists a semi-transitive orientation of $G$ where $v$ is a source (or a sink).   
\end{theorem}

\subsection{Hameed's results and our contribution}

Hameed~\cite{Hameed} provided a complete classification of semi-transitive extended Mycielski graphs: 

\begin{theorem}[\cite{Hameed}]\label{Hameed-thm} The graph  $\mu'(G)$ is semi-transitive if and only if $G$ is a bipartite graph. \end{theorem}

Regarding Mycielski graphs, Hameed~\cite{Hameed} proved the following result and formulated the following conjecture:

\begin{theorem}[\cite{Hameed}]\label{comp-myc-thm} Let $G$ be a comparability graph. Then $\mu(G)$ is semi-transitive if and only if $G$ is bipartite. \end{theorem}

\begin{conjecture}[\cite{Hameed}]\label{conj} For every graph $G$ the graph $\mu(G)$ is semi-transitive if and only if $G$ is a bipartite graph. \end{conjecture}
 
In this note, we prove Conjecture~\ref{conj} and consequently generalize Theorem~\ref{comp-myc-thm}. Note that the proof of Theorem~\ref{Hameed-thm} is based on proving that $\mu'(C_{2k+1})$ is non-semi-transitive for all $k \ge 1$.  Similarly, our proof of Conjecture~\ref{conj} relies on demonstrating the non-semi-transitivity of $\mu(C_{2k+1})$ for all $k\geq 1$, which was also conjectured in~\cite{Hameed}. Notably, our relatively concise proof can be directly applied to establish the non-semi-transitivity of $\mu'(C_{2k+1})$, which originally required a lengthier consideration involving four cases. 

\section{Classification of semi-transitive Mycielski graphs}\label{sec2}

The following theorem is crucial in proving our main result.

\begin{theorem}\label{ext-Mys-odd-cycle-thm}
The graph $\mu(C_{2k+1})$ is non-semi-transitive for all $k \ge 1$.
\end{theorem}

\begin{proof} 
Recall that the vertices of $C_{2k+1}$ are labelled $1,2,\ldots,2k+1$ around the cycle. Suppose $\mu(C_{2k+1})$ admits a semi-transitive orientation. Then, by Theorem~\ref{source-thm}, we can assume that the vertex $x$ is a source. Then, to avoid shortcuts involving $x$, for each vertex $i\in\{1,\ldots,2k+1\}$, either $(i-1)'\rightarrow i \leftarrow (i+1)'$ or  $(i-1)'\leftarrow i \rightarrow (i+1)'$, where $0'=(2k+1)'$ and $(2k+2)'=1'$. Colour a vertex $i$ red in the former case and blue in the latter case. 
Let us show the following properties of this colouring. 

\begin{itemize}
\item[(1)] If a blue vertex $b$ is adjacent to a red vertex $r$ in $C_{2k+1}$, then the edge is $b\rightarrow r$, i.e. it goes from $b$ to $r$. \\ \\ Indeed, if, say, $b=i$ and $r=(i+1)$ and $(i+1)\rightarrow i$, then $xi'(i+1)i(i+1)'$ is a shortcut contradicting the assumption.
\item[(2)] If $i\rightarrow (i+1)\rightarrow (i+2)$  in $C_{2k+1}$, then the vertex $i$ is blue and the vertex $(i+2)$ is red. \\ \\ Indeed, in any other case, the vertices in $\{i,(i+1),(i+2),(i+1)'\}$ form either a shortcut or a directed cycle. \\ 
Clearly, the similar property holds if  $(i+2)\rightarrow (i+1)\rightarrow i$  in $C_{2k+1}$.
\item[(3)] If vertices $i$ and $(i+m)$ in  $C_{2k+1}$ are of the same colour, and all $(i+1), \ldots, (i+m-1)$ are of the other colour, then $m$ is even. \\ \\
Indeed, let $i$ and $i+m$ be red, and all $(i+1), \ldots, (i+m-1)$ be blue (the other case is considered similarly). If $m=2$ the property holds immediately; so, assume $m\ge 3$. From property (1), we have edges $(i+1) \rightarrow i$ and $(i+m-1) \rightarrow (i+m)$. Since there are edges $(i+2)  \rightarrow (i+1)'\rightarrow i$, we must have the edge $(i+2) \rightarrow (i+1)$ to avoid the shortcut $(i+1)(i+2)(i+1)'i$. Similarly, we must have the edge $(i+m-2) \rightarrow (i+m-1)$. From property (2), each vertex among the vertices $(i+2),\ldots, (i+m-2)$
must be either a source or a sink  in $C_{2k+1}$. In particular, $(i+2)$ is a source. Since sources and sinks must alternate in the cycle, all vertices with even addends $(i+2), (i+4), \ldots$ must be sources, and those with odd addends $(i+3), (i+5), \ldots$ must be sinks. But since $(i+m-2)$ cannot be a sink, $m$ must be even.
\end{itemize}
Note that property (3) holds even if $m=2k+1$, i.e.\ if $i$ and $(i+m)$ coincide.  Let us call each maximal sequence of the same coloured vertices in $C_{2k+1}$ a ``pattern''.  Note that in any orientation of $C_{2k+1}$ there must be a vertex that is neither a source nor a sink.
So, by property~(2), there are at least two patterns. 
It follows from property (3) that each pattern contains an odd number of vertices.  Since the colours of the patterns alternate, the total number of the patterns in the cycle must be even. But then the total number of vertices in the cycle is also even as the sum of an even number of odd addends -- a contradiction, because the cycle $C_{2k+1}$ is odd. Therefore, no semi-transitive orientation of  $\mu(C_{2k+1})$ can exist.
\end{proof}

\begin{remark} Our proof of Theorem~\ref{ext-Mys-odd-cycle-thm} applies without modification if $\mu(C_{2k+1})$ is replaced by $\mu'(C_{2k+1})$, thereby offering an alternative, shorter proof of the respective result in \cite{Hameed}.\end{remark}

Indeed, it is easy to verify that all shortcuts mentioned in the proof of Theorem~\ref{ext-Mys-odd-cycle-thm} remain shortcuts also in the graph $\mu'(C_{2k+1})$.

The following theorem confirms Conjecture~\ref{conj}, and its proof follows the steps presented in the proof of Theorem~4 in \cite{Hameed}. 

\begin{theorem}\label{main-thm} The graph  $\mu(G)$ is semi-transitive if and only if $G$ is a bipartite graph. \end{theorem}

\begin{proof} 
Suppose that $G$ is {\em not} a bipartite graph. Then $G$ must contain an odd cycle. A minimal odd cycle in $G$ is an induced odd cycle $C_{2k+1}$ for some $k\geq 1$. Therefore, $\mu(G)$ contains $\mu(C_{2k+1})$ as an induced subgraph. By Theorem~\ref{ext-Mys-odd-cycle-thm},  $\mu(G)$ is not semi-transitive.

Now suppose $G$ is a bipartite graph on $n$ vertices. Orient $G$ transitively from one part to the other, ensuring  the longest directed path in such an orientation is of length 1. Extend this transitive orientation of $G$ to a semi-transitive orientation of $\mu(G)$ by letting $x$ be a source and orienting edges $z\rightarrow y'$ for all $z,y\in\{1,2,\ldots,n\}$. Clearly, this orientation  is acyclic. Since the longest directed path in such an orientation has length 2 there can be no shortcuts. Hence,  $\mu(G)$ is semi-transitive in this case. \end{proof}

\section*{Acknowledgments} The work of the second author was supported by the state contract of the Sobolev Institute of Mathematics (project FWNF-2022-0019).


\begin{thebibliography}{20}
	
\bibitem{B} E.Z. Bidine, T. Gadi and M. Kchikech. Independence number and packing coloring of generalized Mycielski graphs. 
{\em Discuss. Math. Graph Theory} {\bf 41} (2021), no. 3, 725--747. 
		
\bibitem{Hameed} H. Hameed. On semi-transitivity of (extended) Mycielski graphs, {\em Discr. Appl. Math.} {\bf 359} (2024) 83--88.
				
\bibitem{KL15} S. Kitaev and V. Lozin. Words and Graphs, {\em Springer}, 2015.

\bibitem{KS} S. Kitaev and H. Sun. Human-verifiable proofs in the theory of word-representable graphs, {\em RAIRO -- Theoretical Informatics and Appl.} {\bf 58} (2024), Special Issue: Randomness and Combinatorics - Edited by Luca Ferrari \& Paolo Massazza, Art. 10, 10pp..
		
\bibitem{M} J. Mycielski. Sur le coloriage des graphes, {\em Colloq. Math.} {\bf 3} (1955), 161--162.
		
		
	\end{thebibliography}
\end{document}